\documentclass[12pt,a4paper]{article}
\usepackage{latexsym,amssymb,amsfonts,amsmath,amsthm,nccmath,enumitem,graphicx,float}
\usepackage{colortbl}
\usepackage[usenames,dvipsnames]{xcolor}
\usepackage[hidelinks]{hyperref}
\usepackage[margin=2cm]{geometry}

\newtheorem{theorem}{Theorem}
\newtheorem{corollary}[theorem]{Corollary}
\newtheorem{lemma}[theorem]{Lemma}\newtheorem{question}[theorem]{Question}

\theoremstyle{definition}

\newcommand{\Z}{\mathbb{Z}}
\newcommand{\B}{\mathcal{B}}
\newcommand{\A}{\mathcal{A}}
\newcommand{\T}{\mathcal{T}}
\renewcommand{\S}{\mathcal{S}}
\newcommand{\F}{\mathcal{F}}
\newcommand{\C}{\mathcal{C}}

\newcommand{\fF}{\mathbf{F}}
\newcommand{\fK}{\mathbf{K}}
\newcommand{\fS}{\mathbf{S}}

\def \leq {\leqslant}
\def \geq {\geqslant}

\def \mod#1{{\:({\rm mod}\ #1)}}
\DeclareMathOperator{\Flim}{Flim}

\let\oldproofname=\proofname
\renewcommand{\proofname}{\rm\bf{\oldproofname}}

\title{Countable homogeneous Steiner triple systems avoiding specified subsystems}

\author{Daniel Horsley\thanks{School of Mathematics, Monash University, Clayton VIC 3800, Australia ({\tt danhorsley@gmail.com})}\ \ and\ Bridget S. Webb\thanks{School of Mathematics and Statistics, The Open University, Milton Keynes MK7 6AA, United Kingdom ({\tt bridget.webb@open.ac.uk})}}
\date{}

\begin{document}
\def\baselinestretch{1.1}\small\normalsize
\maketitle

\begin{abstract}
In this article we construct uncountably many new homogeneous locally finite Steiner triple systems of countably infinite order as Fra\"{\i}ss\'{e} limits of classes of finite Steiner triple systems avoiding certain subsystems.
The construction relies on a new embedding result: any finite partial Steiner triple system has an embedding into a finite Steiner triple system that contains no nontrivial proper subsystems that are not subsystems of the original partial  system.
Fra\"{\i}ss\'e's construction and its variants are rich sources of examples that are central to model-theoretic classification theory, and recently
infinite Steiner systems obtained via Fra\"{\i}ss\'e-type constructions have received attention from the model theory community.
\end{abstract}


\section{Introduction}

A \emph{Steiner triple system} is a pair $(V,\B)$ where $V$ is a set of \emph{points} and $\B$ is a collection of 3-element subsets of $V$ (\emph{blocks}) such that every pair of points occurs together in exactly one block. The \emph{order} of $(V,\B)$ is $|V|$. In this paper we will be concerned with systems with finite or countably infinite order and will refer to them as \emph{finite} or \emph{countably infinite} accordingly. It is well known that a Steiner triple system of finite order $v$ exists if and only if $v \equiv 1 \text{ or } 3 \mod{6}$; such values of $v$ are called \emph{admissible}. An \emph{isomorphism} from a Steiner triple system $(V_1,\B_1)$ to another $(V_2,\B_2)$ is a bijection $f: V_1 \rightarrow V_2$ such that $\{x,y,z\} \in \B_1$ if and only if $\{f(x),f(y),f(z)\} \in \B_2$.

If $(V',\B')$ and $(V,\B)$ are Steiner triple systems such that $V' \subseteq V$ and $\B' \subseteq \B$, then we say that $(V',\B')$ is a \emph{subsystem} of $(V,\B)$. The subsystem $(V',\B')$ is \emph{proper} if $V' \neq V$. We say a system or subsystem is \emph{nontrivial} if it has order strictly greater than 3. Doyen~\cite{Doy69} proved that there is a Steiner triple system of each admissible order that has no nontrivial proper subsystems. We call such systems \emph{subsystem-free} (while remembering that every Steiner triple system has trivial subsystems and is a subsystem of itself).

The major motivation for this paper comes from Fra\"{\i}ss\'{e}'s theorem~\cite{Fraisse}, an important result in model theory. We direct readers to~\cite{WH} for a formal statement. Here, we instead content ourselves with a brief simplified overview. Suppose we represent mathematical structures of a certain type (for example, groups, graphs or Steiner triple systems) in a consistent way so that each consists of a domain of elements, some of them perhaps distinguished as special constants, together with some functions and/or relations on that domain. This representation gives rise to a definition of isomorphism for our structures and, importantly, a definition of a \emph{substructure} in one of our structures. A substructure $A$ of a structure $B$ is said to be \emph{finitely generated} if there is some finite subset $X$ of the domain of $B$ such that $A$ is the minimal (with respect to subset inclusion on domains) substructure of $B$ whose domain contains $A$. A countable structure is \emph{homogeneous} if every isomorphism between two of its finitely generated substructures can be extended to an automorphism of the entire structure. We say a structure $A$ has \emph{age} $\mathbf{J}$ if the structures in $\mathbf{J}$ are, up to isomorphism, exactly the finitely generated substructures of $A$. Fra\"{\i}ss\'{e}'s theorem states that, for a nonempty class $\fK$ of finitely generated structures, there is a unique (up to isomorphism) countable homogeneous structure $\Flim(\fK)$ that has age $\fK$, provided that $\fK$ obeys the following four properties.
\begin{description}
    \item[Essential countability.]
Up to isomorphism, $\fK$ contains countably many structures.
    \item[Hereditary property.]
If $B \in \fK$ and $A$ is a finitely generated substructure of $B$, then $A$ is isomorphic to some structure in $\fK$.
    \item[Joint embedding property.]
If $B,C \in \fK$, then there is a $D \in \fK$ that contains a substructure isomorphic to $B$ and a substructure isomorphic to $C$.
    \item[Amalgamation property.]
If $A,B,C \in \fK$, and there are isomorphisms $f$ and $g$ from $A$ to substructures of $B$ and $C$ respectively, then there is a $D \in \fK$ and isomorphisms $f'$ and~$g'$ from $B$ and $C$ respectively to substructures of $D$ such that $f' \circ f = g' \circ g$.
\end{description}
The class $\fK$ is called an \emph{amalgamation class} and $\Flim(\fK)$ is called the Fra\"{\i}ss\'{e} limit of $\fK$. Furthermore, if all of the structures in $\fK$ are finite then $\Flim(\fK)$ will be \emph{locally finite}: every finite subset of its domain will be contained in the domain of one of its finite substructures.

There are several reasonable ways to represent Steiner triple systems. In this paper we will view them as \emph{functional} structures. We discuss the details of this representation, along with alternatives to it and prior work that has concerned them in Section~\ref{S:homog}. For now, however, the important upshot of this functional representation is that it implies that the relevant substructures of Steiner triple systems for use in Fra\"{\i}ss\'{e}'s theorem will be subsystems as defined above (and also that isomorphisms of Steiner triple systems will be as defined above). Knowing that `substructure' should be interpreted as `subsystem' gives us definitions of homogeneous and locally finite Steiner triple systems and of finitely generated subsystems.

In this article we construct new countably infinite homogeneous Steiner triple systems as the Fra\"{\i}ss\'{e} limit of classes of  finite Steiner triple systems avoiding certain subsystems. For a class $\fF$ of finite nontrivial Steiner triple systems, we say that a Steiner triple system $(V,\B)$ is \emph{$\fF$-free} if no subsystem of $(V,\B)$ is isomorphic to a system in $\fF$. We call $\fF$ \emph{good} if there exists a finite
$\fF$-free Steiner triple system that is not isomorphic to any subsystem of a system in $\fF$.
In particular, if there is a subsystem-free nontrivial Steiner triple system which is not isomorphic to any subsystem of a system in $\fF$, then $\fF$ can be seen to be good by considering that system. This implies that any finite $\fF$ is good because there exists a subsystem-free Steiner triple system whose order is greater than that of any system in $\fF$.

\begin{theorem}\label{T:amalgClass}
For any good class $\fF$ of finite nontrivial Steiner triple systems, the class $\fK$ of all finite $\fF$-free Steiner triple systems forms an amalgamation class with countably infinitely many nonisomorphic elements. Hence the Fra\"{\i}ss\'{e} limit of $\fK$ is a homogeneous locally finite Steiner triple system of countably infinite order.
\end{theorem}

We show in Section~\ref{S:proofTh1-2} that, by taking $\fF$ in the above theorem to be various subclasses of the class of all subsystem-free Steiner triple systems, we can obtain uncountably many nonisomorphic countable homogeneous locally finite Steiner triple systems.

\begin{corollary}\label{C:uncountable}
There are exactly $2^{\aleph_0}$ non-isomorphic homogeneous locally finite Steiner triple systems of countably infinite order.
\end{corollary}

Most of the work in proving Theorem~\ref{T:amalgClass} is in establishing that the classes satisfy the joint embedding and amalgamation properties. We do this using Theorem~\ref{T:embedNoNewSubsystems}, a new result concerning embeddings of partial Steiner triple systems. A \emph{partial Steiner triple system} is a pair $(U,\A)$ where $U$ is a set of \emph{points} and $\A$ is a collection of 3-element subsets of $U$ (\emph{blocks}) such that every pair of points occurs together in \emph{at most} one block. The \emph{order} of $(U,\A)$ is $|U|$. The \emph{leave} of $(U,\A)$ is the graph $L$ with vertex set $U$ and edge set given by $\{x,y\} \in E(L)$ if and only if $x$ and $y$ occur together in no block in $\A$. Let $(U',\A')$ and $(U,\A)$ be partial Steiner triple systems such that $U' \subseteq U$ and $\A' \subseteq \A$. We say that $(U',\A')$ is \emph{embedded} in $(U,\A)$. Furthermore, if $(U',\A')$ is a (complete) Steiner triple system, then we say that $(U',\A')$ is a \emph{subsystem} of $(U,\A)$. This extends our earlier usage by allowing partial Steiner triple systems to have subsystems. Note that throughout this paper we use the term embedding in the design-theoretic sense just defined rather than in its model-theoretic sense. Our embedding result shows we can embed a finite partial Steiner triple system in a finite complete Steiner triple system without creating any new subsystems.

\begin{theorem}\label{T:embedNoNewSubsystems}
Any finite partial Steiner triple system $(U,\A)$ has an embedding in a finite (complete) Steiner triple system $(V,\mathcal{B})$ that contains no nontrivial proper subsystems that are not subsystems of $(U,\A)$.
\end{theorem}

In the next section we give the necessary background and definitions on homogeneous Steiner triple systems. In Section~\ref{S:proofTh1-2} we show that Theorems~\ref{T:amalgClass} and \ref{T:embedNoNewSubsystems} follow quite easily from a key lemma, Lemma~\ref{L:PSTSEmbedding}. Then, after some preliminaries in Section~\ref{L:PSTSEmbedding}, we will finally prove  Lemma~\ref{L:PSTSEmbedding} in Section~\ref{S:lemmaProof}.

\section{Homogeneity and Steiner triple systems}\label{S:homog}

As discussed briefly in the introduction, there are multiple ways to represent Steiner triple systems from a model-theoretic perspective. These different representations give rise to different notions of a substructure and hence different classifications of which Steiner triple systems are homogeneous. For more on the model theory of Steiner triple systems see~\cite{baldwin,silvia}.

In this paper we are viewing Steiner triple systems as \emph{functional} structures. More specifically we view a Steiner triple system $(V,\B)$ as a structure on domain $V$ with a single binary function~$\circ$ defined by $x \circ x = x$ for all $x \in V$ and $x \circ y = z$, where $z$ is the unique element of $V$ such that $\{x,y,z\} \in \B$, for all distinct $x,y \in V$. The resulting pair $(V,\circ)$ is a special kind of quasigroup, sometimes called a \emph{Steiner quasigroup}. As mentioned in the introduction, an important consequence of representing Steiner triple systems in this way is that
the resulting notion of a substructure is exactly our notion of a subsystem.

When viewed from our functional perspective, the finite homogeneous Steiner triple systems can be completely classified. A finite Steiner triple system is homogeneous if and only if it is isomorphic to one of the finite projective or affine triple systems: the systems comprising the points and lines of $PG(n,2)$ and $AG(n,3)$, respectively, for positive integers~$n$. That the finite projective and affine triple systems are homogeneous follows from Witt's Lemma (see \cite[p. 81]{Asc} for example). Furthermore, any homogeneous Steiner triple system must obviously be 2-point transitive and hence isomorphic to a projective or affine triple system by the main result of~\cite{KeyShu}.

The countably infinite homogeneous Steiner triple systems have not been classified, however. Up to isomorphism the only ones that have appeared in the literature to date are the countably infinite projective and affine triple systems and the countable universal homogeneous locally finite Steiner triple system. These can be formed as the Fra\"{\i}ss\'{e} limits of the classes of finite projective systems, finite affine systems, and all finite systems, respectively. The existence of the countable universal homogeneous system was first noted in~\cite{Thomas}. Theorem~\ref{T:amalgClass} provides uncountably many new examples of homogeneous locally finite Steiner triple systems of countably infinite order as the Fra\"{\i}ss\'{e} limits of classes of finite Steiner triple systems avoiding certain subsystems.

The alternative to representing Steiner triple systems functionally, as we do in this paper, is to represent them \emph{relationally} without using functions. Most simply, one can view a Steiner triple system $(V,\B)$ as a structure on domain $V$ with a single ternary relation that relates three points if and only if they form a block in $\B$. Under this representation a substructure of a Steiner triple system is an ``induced'' partial Steiner triple system obtained by choosing some subset of $V$ and deleting the points not in this subset and the blocks that include any of the deleted points. This is obviously a much wider notion of substructure than that given by our definition of subsystem. This wide notion of substructure leads to a very restricted family of homogeneous structures: a nontrivial Steiner triple system is \emph{relationally} homogeneous if and only if it has order 7 and hence is isomorphic to the projective triple system comprising the points and lines of $PG(2,2)$.

We now briefly discuss work that has been done concerning these relational representations. One way of viewing a relationally represented Steiner triple system is as a special kind of 3-uniform hypergraph in which every pair of vertices is contained in exactly one hyperedge. Homogeneous 3-uniform hypergraphs are studied in~\cite{AL,LT}. One can also view a relationally represented Steiner triple system as a special case of a Steiner system represented as a domain of points equipped with a $k$-ary relation such that every $t$-set of vertices is contained in exactly one $k$-set of points defined by the relation. Homogeneous Steiner systems, viewed in this way, are classified in~\cite{Dev}. Finally one can also view a relationally represented Steiner triple system as a special case of a linear space represented as a domain of points equipped with a ternary collinearity relation. Homogeneous linear spaces are classified in~\cite{DD}. (Note that when viewing Steiner triple systems as linear spaces, substructures would be viewed as complete linear subspaces where any two points that do not appear together in a collinear triple are assumed to be in a line of length 2. Of course, this does not change the classification of homogeneity.)

It should be noted that the term \emph{ultrahomogeneity} is sometimes used in the literature to describe our property \emph{homogeneity}, for example in~\cite{Dev,DD}. In this case, the term homogeneous is sometimes used to describe a weaker, but similar, property: a structure is \emph{set-homogeneous} if whenever two finitely generated substructures are isomorphic there is some automorphism of the whole structure mapping one to the other.
For finite graphs these two properties are equivalent~\cite{enomoto,ronse}, but this is not the case in general. For example, the equivalence does not hold for countably infinite graphs~\cite{DGMS94}. It also does not hold for Steiner triple systems considered either as relational or functional structures. The only finite \emph{relationally} set-homogeneous, but not homogeneous, Steiner triple system has order~9 and hence is isomorphic to the affine triple systems comprising the points and lines of $AG(2,3)$~\cite{Dev,DD}.
Finite \emph{functionally} set-homogeneous Steiner triple systems are considered in~\cite{DarBarB}.

In addition, the term \emph{homogeneous} is sometimes used with completely different meanings when discussing similar structures, for example, quasigroups~\cite{Menon} and Latin squares~\cite{KD}.

\section{Proofs of main results from Lemma~\ref{L:PSTSEmbedding}}\label{S:proofTh1-2}

The following lemma is crucial to our argument.

\begin{lemma}\label{L:PSTSEmbedding}
Let $(U,\A)$ be a finite partial Steiner triple system of odd order $u \geq 11$ whose leave contains a $6$-cycle $H$. There exists a partial Steiner triple system $(V,\A \cup \B)$ of order $2u+1$ such that $(V,\A \cup \B)$ contains no nontrivial proper subsystems that are not subsystems of $(U,\A)$ and the edge set of the leave of $(V,\A \cup \B)$ is obtained from the edge set of the leave of $(U,\A)$ by deleting the edges of the cycle $H$.
\end{lemma}

In this section we show that Lemma~\ref{L:PSTSEmbedding} implies Theorem~\ref{T:embedNoNewSubsystems}, which in turn implies Theorem~\ref{T:amalgClass}, which itself finally implies Corollary~\ref{C:uncountable}. Given these implications, it will then only remain to prove Lemma~\ref{L:PSTSEmbedding}, and we will do this over the next two sections.

\begin{proof}[\textbf{\textup{Proof of Theorem~\ref{T:embedNoNewSubsystems} from Lemma~\ref{L:PSTSEmbedding}.}}]
Let $U'$ be a superset of $U$ such that $|U'|$ is large compared to $|U|$, $|U'| \equiv 1,9 \mod{12}$ if $|\A|$ is even and $|U'| \equiv 3,7 \mod{12}$ if $|\A|$ is odd. Let $L'$ be the leave of the partial Steiner triple system $(U',\A)$. It is routine to check that $|E(L')|\equiv 0 \mod 6$ and that each vertex of $L'$ has even degree close to $|U'|$. Thus, by the main result of~\cite{KuhnOsthus}, there is a decomposition $\{H_1,\ldots,H_t\}$ of $L'$ into $6$-cycles. Let $(U_0,\A_0)=(U',\A)$. Define a sequence of partial Steiner triple systems $(U_0,\A_0),\ldots,(U_t,\A_t)$ such that, for $i \in \{1,2,\ldots,t\}$, $(U_i,\mathcal{A}_i)$ is the partial Steiner triple system obtained from $(U_{i-1},\A_{i-1})$ via Lemma \ref{L:PSTSEmbedding} with $H$ chosen to be $H_i$. Then, for each $i \in \{0,\ldots,t\}$, $(U_i,\A_i)$ contains no nontrivial proper subsystems that are not subsystems of $(U,\A)$ and the edge set of the leave of $(U_i,\A_i)$ is $\bigcup_{j=i+1}^tE(H_j)$. In particular, $(U_t,\A_t)$ is a (complete) Steiner triple system that contains no nontrivial proper subsystems that are not subsystems of $(U,\A)$.
\end{proof}

\begin{proof}[\textbf{\textup{Proof of Theorem~\ref{T:amalgClass} from Theorem~\ref{T:embedNoNewSubsystems}.}}]
Fix a good set $\fF$ of finite Steiner triple systems and let $\fK$ be the set of all finite
$\fF$-free Steiner triple systems. Because every system in $\fK$ is finite, $\fK$ contains countably many structures. Also, $\fK$ has the hereditary property because if $(V,\B)$ is an $\fF$-free Steiner triple system and $(V',\B')$ is a subsystem of $(V,\B)$, then $(V',\B')$ is also $\fF$-free.

To establish the joint embedding and amalgamation properties we proceed as follows. For the joint embedding property, given two systems in $\fK$, we take $(V_1,\B_1)$ and $(V_2,\B_2)$ to be isomorphic copies of these systems such that $V_1 \cap V_2 = \emptyset$. For the amalgamation property, given two systems in $\fK$ each with a specified subsystem isomorphic to a third system in $\fK$, we take $(V_1,\B_1)$ and $(V_2,\B_2)$ to be isomorphic copies of the two systems such that $(V_1 \cap V_2,\B_1 \cap \B_2)$ is the specified subsystem in each. In either case it can be seen that it suffices to embed $(V_1 \cup V_2,\B_1 \cup \B_2)$ in a finite $\fF$-free Steiner triple system.

Because $\fF$ is good, there is a finite $\fF$-free Steiner triple system $(V^*,\B^*)$ such that $(V^*,\B^*)$ is not isomorphic to any subsystem of a system in $\fF$ and $V^*$ is disjoint from $V_1 \cup V_2$. Let $(U,\A)$ be the partial Steiner triple system $(V_1 \cup V_2 \cup V^*,\B_1 \cup \B_2 \cup \B^*)$. By Theorem~\ref{T:embedNoNewSubsystems}, $(U,\A)$ has an embedding in a (complete) Steiner triple system $(V,\mathcal{B})$ that contains no nontrivial proper subsystems that are not subsystems of $(U,\A)$. Thus, because $(V_1,\B_1)$, $(V_2,\B_2)$ and $(V^*,\B^*)$ are $\fF$-free, no proper subsystem of $(V,\mathcal{B})$ is in $\fF$. Furthermore $(V,\mathcal{B})$ itself is not in $\fF$ because it has $(V^*,\B^*)$ as a subsystem. Thus $(V,\mathcal{B})$ is $\fF$-free, and we have that $\fK$ obeys the joint embedding and amalgamation properties. Furthermore, because $|V| > \max(|V_1|,|V_2|)$, $\fK$ contains systems of infinitely many orders. Thus $\fK$ is an amalgamation class with countably infinitely many
nonisomorphic elements. The remainder of the theorem follows from Fra\"{\i}ss\'{e}'s theorem~\cite{Fraisse,WH}.
\end{proof}

\begin{proof}[\textbf{\textup{Proof of Corollary~\ref{C:uncountable} from Theorem~\ref{T:amalgClass}.}}]
Consider the set of all isomorphism classes of nontrivial finite subsystem-free Steiner triple systems. Let $\fS^*$ be a set containing exactly one representative of each isomorphism class in this set. Now $|\fS^*|=\aleph_0$ because, by the result of Doyen~\cite{Doy69} mentioned in the introduction, it contains at least one system of each admissible order greater than 3. Let $\fF$ be one of the $2^{\aleph_0}$ proper subsets of $\fS^*$.
Because $\fF$ is a proper subset of $\fS^*$, it follows that $\fF$ is good. Hence by Theorem~\ref{T:amalgClass}, the class $\fK$ of all finite \mbox{$\fF$-free} Steiner triple systems is an amalgamation class and its  Fra\"{\i}ss\'{e} limit $\Flim(\fK)$ is a homogeneous locally finite Steiner triple system of countably infinite order. So it suffices to show that any two of the $2^{\aleph_0}$ possible choices for $\fF$ lead to two nonisomorphic Fra\"{\i}ss\'{e} limits.

Let $\fF_1$ and $\fF_2$ be two distinct proper subsets of $\fS^*$ and suppose without loss of generality that $\fF_1 \setminus \fF_2$ is nonempty. Let $\fK_1$ and~$\fK_2$ be the amalgamation classes of all \mbox{$\fF_1$-free} and all \mbox{$\fF_2$-free} Steiner triple systems, respectively.
Then $\Flim(\fK_1)$ and $\Flim(\fK_2)$ are not isomorphic because any system in $\fF_1 \setminus \fF_2$ is \mbox{$\fF_2$-free} and hence is isomorphic to a subsystem of $\Flim(\fK_2)$ but not to a subsystem of $\Flim(\fK_1)$.
\end{proof}

\section{Preliminaries for the proof of Lemma~\ref{L:PSTSEmbedding}}

In this section we give some definitions and preliminary results that will aid us in our proof of Lemma~\ref{L:PSTSEmbedding}.

For an integer $n \geq 2$ we let $\Z_n$ denote the additive group of integers modulo $n$. For an element $x$ of $\Z_n$, we abbreviate $x+x$ to $2x$ and so on. A \emph{$1$-factor} of a graph $G$ is a set $F$ of edges of $G$ such that each vertex of $G$ is incident with exactly one edge in $F$. A \emph{$1$-factorisation} of a graph $G$ is a nonempty set of 1-factors of $G$ which partition its edge set. Here we will only be interested in 1-factorisations of complete graphs and will refer to a 1-factorisation of the complete graph on vertex set $V$ as simply a  \emph{$1$-factorisation on $V$}. Of course, $|V|$ must be even for a $1$-factorisation on $V$ to exist.
If $\F$ is a 1-factorisation on $V$ and $S$ is a nonempty subset of $V$, then we say that $S$ \emph{induces a sub-$1$-factorisation of $\F$} if there is a 1-factorisation $\F'$ on $S$ such that each 1-factor in $\F'$ is a subset of a 1-factor in $\F$. For odd $n \geq 3$, the \emph{standard $1$-factorisation} on $\Z_n \cup \{\infty\}$ is given by $\{F_0,\ldots,F_{n-1}\}$ where, for $i \in \{0,\ldots,n-1\}$,
$$F_i=\{\{x,y\}:x,y \in \Z_n, x+y=2i\} \cup \{\{\infty,i\}\}.$$

We first show that sometimes the closure of a subset of $\Z_n$ under addition of distinct elements is sufficient to ensure that it forms a subgroup.

\begin{lemma}\label{L:subgroup}
Let $n$ be an odd positive integer and let $S$ be a subset of $\Z_n$ such that $|S| \geq 3$, $0 \in S$, and $a+b \in S$ for any distinct elements $a$ and $b$ of $S$. Then $S$ forms a subgroup of $\Z_n$.
\end{lemma}

\begin{proof}
First we make the following observation. For any $x \in S$, if $2x \in S$ then, by repeatedly adding $x$ to $2x$, we have that $kx \in S$ for each integer $k$ and hence that $-x \in S$. In particular, to prove the lemma it suffices to show that $2x \in S$ for each $x \in S$.

Let $a$ and $b$ be distinct nonzero elements of $S$. Then we can deduce that each of $a+b$, $(a+b)+a=2a+b$, $(a+b)+b=a+2b$ and $(2a+b)+b=2(a+b)$ is in $S$, noting that $2a+b \neq b$ because $2a \neq 0$ since $n$ is odd. Because $2(a+b) \in S$, we have $-(a+b) \in S$ by our observation above. Now $-(a+b)$ cannot be equal to both $a$ and $b$, and so we may assume that $-(a+b) \neq a$ without loss of generality. Then $a-(a+b)=-b \in S$ and hence $(2a+b)-b=2a \in S$, noting that $2a+b \neq -b$ because $2(a+b) \neq 0$ since $n$ is odd. Then $-a \in S$ by our observation and hence $(a+2b)-a=2b \in S$, noting that $a+2b \neq -a$ because $2(a+b) \neq 0$. So we have seen that $2a$ and $2b$ are in $S$ and, since $a$ and $b$ were arbitrary distinct nonzero elements of $\Z_n$, we have shown that $2x \in S$ for each $x \in S$ as required.
\end{proof}

We are now in a position to characterise the subsets of $\Z_n \cup \{\infty\}$ that induce a sub-\mbox{1-factorisation} of the standard $1$-factorisation.

\begin{lemma}\label{L:sub1F}
Let $n$ be an odd integer and let $\{F_0,\ldots,F_{n-1}\}$ be the standard 1-factorisation on $\Z_n \cup \{\infty\}$. Suppose that $S$ is a subset of $\Z_n \cup \{\infty\}$ with $|S|>2$ that induces a \mbox{sub-1-factorisation} $\F'$ of $\{F_0,\ldots,F_{n-1}\}$. Then $S$ is the union of $\{\infty\}$ and a coset $C$ of a subgroup of $\Z_n$ and $\F'=\{F'_i:i \in C\}$ where $F'_i=\{\{x,y\} \in F_i:x,y \in S\}$ for each $i \in \Z_n$.
\end{lemma}
\begin{proof}
We know that $|S|$ is even because it induces a \mbox{sub-1-factorisation} and hence we have that $|S| \geq 4$. By the 1-rotational symmetry of the standard 1-factorisation it suffices to show that if $0 \in S$, then $S$ is the union of $\{\infty\}$ and a subgroup of $\Z_n$. Let $x$ and $y$ be distinct elements of $S \cap \Z_n$. Then $\{x,y\} \in F'_z$ where $z$ is the unique element of $\Z_n$ such that $2z=x+y$. Thus each edge of $F_z$ incident to a vertex in $S$ is in $F'_z$ and so, because $\{0,x+y\} \in F_z$ and $0 \in S$, we must have $x+y \in S$. Thus $S \cap \Z_n$ is closed under addition of distinct elements and so it forms a subgroup of $\Z_n$ by Lemma~\ref{L:subgroup}. In particular, $|S \cap \Z_n|$ is odd because $n$ is odd. However, $|S|$ is even and hence it must be that $\infty \in S$. The result now follows.
\end{proof}

The following easy observation will be useful. Roughly speaking it says that $r$ points of a partial Steiner triple system are sufficient to uniquely determine a subsystem of order at most $2r$.

\begin{lemma}\label{L:uniqueSubsystem}
Let $(U,\mathcal{A})$ be a partial Steiner triple system. For any subset $R$ of $U$, there is at most one subsystem of $(U,\mathcal{A})$ whose point set contains $R$ and whose order is at most $2|R|$.
\end{lemma}

\begin{proof}
Suppose otherwise that $(S_1,\T_1)$ and $(S_2,\T_2)$ are two distinct subsystems of $(U,\mathcal{A})$ such that, for $i \in \{1,2\}$, we have $R \subseteq S_i$ and $|S_i| \leq 2|R|$. Then $(S', \T')$ where $S'=S_1 \cap S_2$ and $\T'= \T_1 \cap \T_2$ is a subsystem of both $(S_1,\T_1)$ and $(S_2,\T_2)$. Note that $|R| \leq |S'|$ because $R \subseteq S_i$ for $i \in \{1,2\}$. Since $(S_1,\T_1)$ and $(S_2,\T_2)$ are distinct, we can assume without loss of generality that $(S', \T')$ is a proper subsystem of $(S_1,\T_1)$. But then $|S_1| \leq 2|R| \leq 2|S'|$ and hence $(S_1,\T_1)$ has a subsystem of order at least $\frac{1}{2}|S_1|$, which contradicts the Doyen-Wilson theorem~\cite{DW}.
\end{proof}

Finally, we will also use the following bound on the sum of a finite subsequence of a generalised harmonic sequence.

\begin{lemma}\label{L:harmonicTypeSum}
Let $d'$ and $d''$ be odd integers such that $3 \leq d' \leq d''$. Then
\[\mfrac{1}{d'}+\mfrac{1}{d'+2}+\mfrac{1}{d'+4}+\cdots+\mfrac{1}{d''}\leq \mfrac{1}{2}\ln
\left(\mfrac{d''+1}{d'-1}\right).\]
\end{lemma}

\begin{proof}
Let $f: [d'-1,d''+1] \rightarrow \mathbb{R}$ be defined by $f(x)=\frac{1}{z}$ where $z=x+1$ if $x$ is an even integer and $z$ is the odd integer closest to $x$ otherwise. Then
\[\mfrac{1}{d'}+\mfrac{1}{d'+2}+\mfrac{1}{d'+4}+\cdots+\mfrac{1}{d''}=\mfrac{1}{2}\medint\int_{d'-1}^{d''+1}f(x)\,dx \leq \mfrac{1}{2}\medint\int_{d'-1}^{d''+1}\mfrac{1}{x}\,dx= \mfrac{1}{2}\ln\left(\mfrac{d''+1}{d'-1}\right),\]
where the first equality follows by considering the region bounded by the graph of $f$ between $d'-1$ and $d''+1$, and the inequality follows because $\int_{d-1}^{d}\frac{1}{x}-\frac{1}{d}\,dx \geq \int_{d}^{d+1}\frac{1}{d}-\frac{1}{x}\,dx$ for each $d \in \{d',d'+2,d'+4,\ldots,d''\}$ (note that $\frac{1}{x}$ is a continuous convex function on the interval $[d'-1,d''+1]$).
\end{proof}

\section{Proof of Lemma~\ref{L:PSTSEmbedding}}\label{S:lemmaProof}

Before we can prove Lemma~\ref{L:PSTSEmbedding}, we require a result that allows us to extend a partial mapping from $\Z_u$ to the point set of a Steiner triple system of order $u$ to a bijection that does not map cosets of subgroups of $\Z_u$ to subsystems of the system. Note that, while we have defined a nontrivial Steiner triple system or subsystem to be one of order greater than 3, we use the conventional definition that a nontrivial subgroup is one of order strictly greater than 1.

\begin{lemma}\label{L:bijectionExistence}
Let $(U,\mathcal{A})$ be a partial Steiner triple system of odd order $u \geq 11$. Then any injection $\varphi':\{0,\ldots,\frac{u+1}{2}\} \rightarrow U$ can be extended to a bijection $\varphi:\Z_u \rightarrow U$ such that, for any coset $C$ of a nontrivial proper subgroup of $\Z_u$, $\varphi(C)$ is not the point set of a (trivial or nontrivial) subsystem of $(U,\mathcal{A})$.
\end{lemma}

\begin{proof}
For a positive divisor $d$ of $u$, let $\langle d \rangle$ denote the subgroup of $\Z_u$ generated by $d$. Let $\C$ be the set of all cosets of nontrivial proper subgroups of $\Z_u$ of order congruent to 1 or 3 modulo 6. For each $s \in \{3,\ldots,u\}$,
let $\S_s$ be the set of all subsets $S$ of $U$ such that $|S|=s$ and $S$ is the point set of a subsystem of $(U,\mathcal{A})$. It will be important throughout the proof that for any $s \in \{3,\ldots,u\}$ and any $(\frac{s+1}{2})$-subset $S'$ of $U$, there is at most one superset of $S'$ in $\S_s$ by Lemma~\ref{L:uniqueSubsystem}.

Let $\varphi': \{0,\ldots,\frac{u+1}{2}\} \rightarrow U$ be an injection.
We wish to construct a bijective extension $\varphi: \Z_u \rightarrow U$ of $\varphi'$ such that $\varphi(C) \notin \S_{|C|}$ for all $C \in \C$. For an injection $\psi: X \rightarrow U$, where $X$ is a subset of $\Z_u$, we say that a coset $C \in \C$ is \emph{safe under $\psi$} if $\psi(C \cap X) \nsubseteq S$ for all $S \in \S_{|C|}$. For each $i \in \{\frac{u+1}{2},\ldots,u-3\}$ we will construct a function $\varphi_{i}$ such that
\begin{enumerate}[label=(\alph*)]
    \item
$\varphi_{i}: \{0,\ldots,i\} \rightarrow U$ is an injective extension of $\varphi'$;
    \item\label{inductiveCond}
for each $C \in \C$ such that $|C \cap \{0,\ldots,i\}| \geq \tfrac{1}{2}(|C|+3)$, $C$ is safe under $\varphi_i$; and
    \item\label{specialCond}
if $u \equiv 3 \mod{6}$ and $i \geq \tfrac{2u}{3}-1$, then there is no $z \in U$ such that
\[\big\{\{\varphi_i(\tfrac{u}{3}-2), \varphi_i(\tfrac{2u}{3}-2), z\},\{\varphi_i(\tfrac{u}{3}-1), \varphi_i(\tfrac{2u}{3}-1), z\}\big\} \subseteq \A.\]
\end{enumerate}

We first show that, if we can construct a $\varphi_{u-3}: \{0,\ldots,u-3\} \rightarrow U$ that obeys (a), (b) and (c), then at least one of the two bijections $\varphi:\Z_u \rightarrow U$ that are extensions of $\varphi_{u-3}$ will have the properties required by the lemma. Note that any coset of a proper subgroup of $\Z_u$ contains at most one of $u-2$ and $u-1$. If $u \not\equiv 3 \mod{6}$ then each $C \in \C$ has size at least 5 and so satisfies $|C \cap \{0,\ldots,u-3\}| \geq \tfrac{1}{2}(|C|+3)$ and is safe under $\varphi_{u-3}$ by (b). Thus both the bijections $\varphi:\Z_u \rightarrow U$ that are extensions of $\varphi_{u-3}$ will have the properties required by the lemma. If $u \equiv 3 \mod{6}$, then the only cosets $C \in \C$ that do not satisfy $|C \cap \{0,\ldots,u-3\}| \geq \tfrac{1}{2}(|C|+3)$ are $\{\tfrac{u}{3}-2,\tfrac{2u}{3}-2,u-2\}$ and $\{\tfrac{u}{3}-1,\tfrac{2u}{3}-1,u-1\}$. Then, using (c) and the fact that the pairs $\{\varphi_{u-3}(\tfrac{u}{3}-2),\varphi_{u-3}(\tfrac{2u}{3}-2)\}$ and $\{\varphi_{u-3}(\tfrac{u}{3}-1),\varphi_{u-3}(\tfrac{2u}{3}-1)\}$ each occur in at most one triple in $\A$, it is not difficult to confirm that one of the two bijections $\varphi:\Z_u \rightarrow U$ that are extensions of $\varphi_{u-3}$ will satisfy $\{\varphi(\tfrac{u}{3}-j), \varphi(\tfrac{2u}{3}-j), \varphi(u-j)\} \notin \A$ for $j \in \{1,2\}$. This bijection will have the properties required by the lemma.

So it only remains to show that, for each $i \in \{\frac{u+1}{2},\ldots,u-3\}$, there is a function $\varphi_{i}$ satisfying (a), (b) and (c). Take $\varphi_{(u+1)/2}=\varphi'$. Obviously $\varphi_{(u+1)/2}$ satisfies (a), and it also satisfies (c) because $\tfrac{2u}{3}-1 > \frac{u+1}{2}$ for $u \geq 11$. Further $\varphi_{(u+1)/2}$ satisfies (b) because, if $C$ is a coset of the subgroup $\langle d \rangle$ for some divisor $d$ of $u$, then $|C \cap \{0,\ldots,\frac{u+1}{2}\}| \leq \frac{u+d}{2d}$
and $\tfrac{1}{2}(|C|+3)=\frac{u}{2d}+\frac{3}{2}$. Let $i$ be a fixed element of $\{\frac{u+3}{2},\ldots,u-3\}$ and suppose inductively that there is a function $\varphi_{i-1}$ satisfying (a), (b) and (c). To complete the proof it suffices to show there is a function $\varphi_{i}$ satisfying (a), (b) and (c).  Let $W$ be the set of all points in $U$ that are not in the image of $\varphi_{i-1}$ and note that $|W|=u-i$. We will define $\varphi_{i}$ as an extension of $\varphi_{i-1}$, so specifying $\varphi_{i}(i)$ will determine it. We only need show that there is at least one choice for $\varphi_{i}(i)$ in $W$ that makes $\varphi_{i}$ satisfy (b) and (c).

Because $\varphi_{i-1}$ satisfies (b), to ensure that $\varphi_{i}$ satisfies (b) it is enough to ensure that each coset $C \in \C_i$ is safe under $\varphi_{i}$, where $\C_i$ is the set of all cosets $C \in \C$ such that $i \in C$ and $|C \cap \{0,\ldots,i\}| = \tfrac{1}{2}(|C|+3)$. If $d$ is a proper divisor of $u$, then the unique coset $C$ of $\langle d \rangle$ containing $i$ intersects $\{0,\ldots,i\}$ in exactly $\tfrac{1}{2}(|C|+3)=\tfrac{1}{2}(\frac{u}{d}+3)$ elements if and only if $i=\tfrac{d}{2}(\frac{u}{d}+1)+j$ for some $j \in \{0,\ldots,d-1\}$ which in turn occurs if and only if $\frac{2i-u+2}{3} \leq d \leq 2i-u$. Let
$$D_i=\{d \in \Z: 1 < d < u,\ d|u,\ \tfrac{u}{d} \equiv 1,3 \mod{6},\ \tfrac{2i-u+2}{3} \leq d \leq 2i-u\}.$$
Note that $|D_i|=|\C_i|$ and that, for each $d \in D_i$, there is exactly one coset $C$ of $\langle d \rangle$ in $\C_i$. Also, for each $C \in \C_i$, there is at most one $S \in \S_{|C|}$ such that $\varphi_{i-1}(C\cap \{0,\ldots,i-1\}) \subseteq S$ by Lemma~\ref{L:uniqueSubsystem} because $|C| \leq 2|C\cap \{0,\ldots,i-1\}|=|C|+1$. Further, if such an $S \in \S_{|C|}$ exists, then $|S \cap W| \leq |C|-\frac{1}{2}(|C|+1) = \frac{1}{2}(\frac{u}{d}-1)$. It follows that our requirement that each coset $C \in \C_i$ be safe under $\varphi_{i}$ forbids at most $r_i$ choices of $\varphi_{i}(i)$ in $W$ where
$$r_i=\mfrac{1}{2}\medop\sum_{d \in D_i}(\mfrac{u}{d}-1).$$

Because $\varphi_{i-1}$ satisfies (c), $\varphi_{i}$ will automatically satisfy (c) unless it is the case that $u \equiv 3 \mod{6}$, $i = \tfrac{2u}{3}-1$ and $\{\varphi_{i-1}(\tfrac{u}{3}-2), \varphi_{i-1}(\tfrac{2u}{3}-2), z\} \in \A$ for some $z \in W$. In this remaining case, $\varphi_{i}$ will satisfy (c) provided that we do not choose $\varphi_{i}(i)$ so that $\{\varphi_{i-1}(\tfrac{u}{3}-1), \varphi_i(i), z\} \in \A$. This forbids at most one choice of $\varphi_{i}(i)$ in $W$.

Thus, because $|W| = u-i$, it suffices to show $r_i<u-i-1$ to establish that at least one choice for $\varphi_{i}(i)$ in $W$ makes $\varphi_{i}$ satisfy (b) and (c), and hence to complete the proof. First suppose that $u \leq 19$. If $u \neq 15$, then $D_i = \emptyset$ and $r_i=0$. If $u=15$, then $D_i \subseteq \{5\}$ and $r_i \leq 1$. So in each case we obviously have $r_i < u-i -1$. Thus we may assume that $u \geq 21$.

The following table details values of $i$ for which we can easily establish that $r_i < u-i -1$. The first column gives a range of values for $i$. In the second column the upper bound for $i$ is used to give a lower bound on $u-i -1$. In the third column the bounds for $i$ and the definition of $D_i$ are used to establish a superset of $D_i$ (in the first two rows, note that $\frac{u}{5} \notin D_i$ because $\tfrac{u}{d} \equiv 1,3 \mod{6}$ for each $d \in D_i$). In the final column the superset of $D_i$ is used to find an upper bound on $r_i$. In each case it is then routine to check that $r_i < u-i -1$ given that $u \geq 21$.
$$
\begin{tabular}{llll}
    $\frac{5u}{7}-1<i\leq u-3$ & $u-i-1 \geq 2$ & $D_i \subseteq \{\frac{u}{3}\}$ & $r_i \leq 1$ \\
    $\frac{2u}{3}-1<i\leq \frac{5u}{7}-1\rule{0.5cm}{0cm}$  & $u-i-1 \geq \frac{2u}{7} \rule{0.5cm}{0cm}$ & $D_i \subseteq \{\frac{u}{7},\frac{u}{3}\}$ & $r_i \leq 4$ \\
    $\frac{u+3}{2} \leq i \leq \frac{u+7}{2}$ & $u-i-1 \geq \frac{u-9}{2}$ & $D_i \subseteq \{3,5,7\}$ & $r_i \leq \frac{71u}{210}-\frac{3}{2}$
\end{tabular}
$$

So we can assume that $\frac{u+9}{2} \leq i \leq \frac{2u}{3}-1$ and hence that $u \geq 33$.  Because $i \leq \frac{2u}{3}-1$, it suffices to show that $r_i < \frac{u}{3}$. If $D_i=\emptyset$, then $r_i=0$ and we obviously have $r_i < u-i-1$, so we can assume that $D_i$ is nonempty. Let $d'$ and $d''$ be the smallest and largest elements of $D_i$ respectively. Since every element of $D_i$ is odd, it can be seen that
\begin{equation}\label{E:rBound}
r_i \leq \mfrac{u}{2}\left(\mfrac{1}{d'}+\mfrac{1}{d'+2}+\mfrac{1}{d'+4}\cdots+\mfrac{1}{d''}\right)-\mfrac{|D_i|}{2} < \mfrac{u}{4}\ln\left(\mfrac{d''+1}{d'-1}\right)
\end{equation}
where the last inequality follows by Lemma~\ref{L:harmonicTypeSum}. We have that $d' \geq \tfrac{2i-u+2}{3}$ and $d'' \leq 2i-u$ from the definition of $D_i$ and thus that $\frac{d''+1}{d'-1} \leq \frac{3(2i-u+1)}{2i-u-1}$. Thus, since $\frac{u+9}{2} \leq i$, we have that $2i-u \geq 9$ and hence that $\frac{d''+1}{d'-1} \leq \frac{15}{4}$ (note that $\frac{3(x+1)}{x-1}$ is a monotonically decreasing function of $x$ for $x \geq 9$). So it follows from \eqref{E:rBound} that $r_i < \frac{u}{4}\ln(\frac{15}{4}) < \frac{u}{3}$ as required.
\end{proof}

We are now in a position to prove Lemma~\ref{L:PSTSEmbedding}.

\begin{proof}[\textbf{\textup{Proof of Lemma~\ref{L:PSTSEmbedding}.}}]
Suppose without loss of generality that $(\Z_u \cup \{\infty\}) \cap U = \emptyset$. Let $V(H)=\{x_1,\ldots,x_6\}$ such that $E(H)=\{x_1x_2,x_2x_3,\ldots,x_5x_6,x_6x_1\}$, and let $\varphi': \{0,\ldots,\frac{u+1}{2}\} \rightarrow U$ be any injection such that
$$\varphi'(1)=x_1\qquad \varphi'(0)=x_2\qquad \varphi'(2)=x_3\qquad \varphi'(4)=x_4\qquad \varphi'(6)=x_5\qquad \varphi'(5)=x_6.$$
By Lemma \ref{L:bijectionExistence}, there is a bijective extension $\varphi: \Z_u \rightarrow U$ of $\varphi'$ such that for any coset $C$ of a nontrivial proper subgroup of $\Z_u$, $\varphi(C)$ is not the point set of a subsystem of $(U,\mathcal{A})$. Let $V = U \cup \Z_u \cup \{\infty\}$ and let $\{F_0,\ldots,F_{u-1}\}$ be the standard 1-factorisation on $\Z_u \cup \{\infty\}$. Let
$$\B^\dag=\bigcup_{i \in \Z_u}\big\{\{x,y,\varphi(i)\}:\{x,y\} \in F_i\big\}.$$
Then $(V,\A \cup \B^\dag)$ is a partial Steiner triple system. Note that the leave of $(V,\A \cup \B^\dag)$ has the same edge set as the leave of $(U,\mathcal{A})$.

Now let
\begin{align*}
  \B^\ddag &= \big\{\{3,u-1,x_1\},\{u-1,1,x_2\},\{1,3,x_3\},\{3,5,x_4\},\{5,7,x_5\},\{7,3,x_6\}\big\}  \\
  \B_0 &= \big\{\{u-1,1,3\}, \{3,5,7\}\big\} \\
  \B_2 &= \big\{\{x_1,x_2,u-1\}, \{x_2,x_3,1\}, \{x_3,x_4,3\}, \{x_4,x_5,5\}, \{x_5,x_6,7\}, \{x_1,x_6,3\}\big\}
\end{align*}
and let $\B_1 =\B^\dag \setminus \B^\ddag$ and $\B=\B_0 \cup \B_1 \cup \B_2$. It is routine to check that $\B^\ddag \subseteq \B^\dag$ using the definitions of $\B^\dag$, $\{F_0,\ldots,F_{n-1}\}$, and $\varphi$. Then $(V,\A \cup \B)$ is a partial Steiner triple system (see Figure~\ref{F:triDecomps}). Note that the edge set of the leave of $(V,\A \cup \B)$ is obtained from the edge set of the leave of $(U,\A)$ by deleting the edges of the cycle $H$. We will complete the proof by showing that $(V,\A \cup \B)$ contains no nontrivial proper subsystems that are not subsystems of $(U,\A)$.

\begin{figure}[H]
\begin{centering}
\includegraphics[scale=0.8]{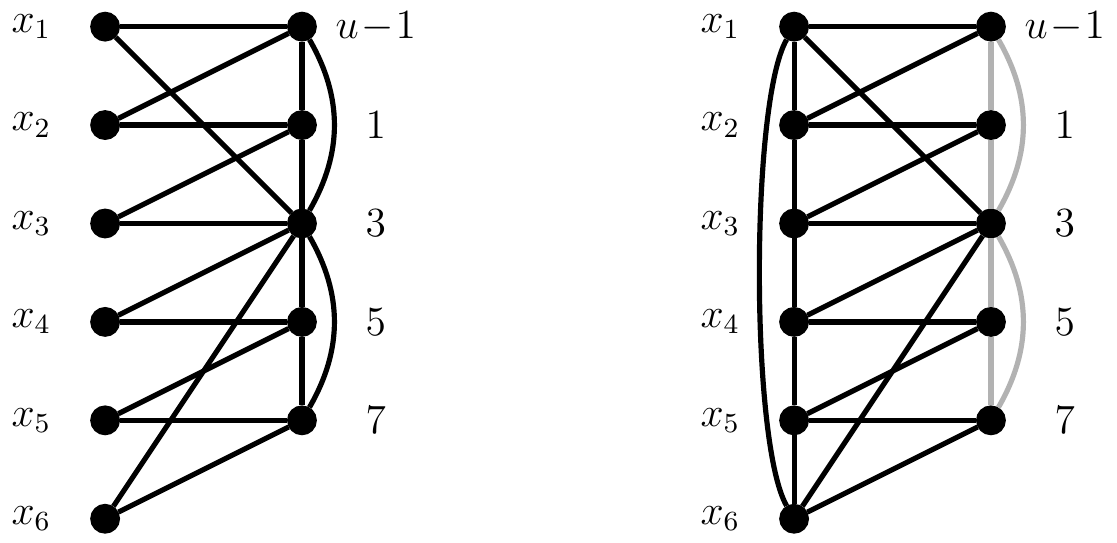}
\caption{\label{F:triDecomps}Graphs whose edges correspond to the pairs occuring in the triples in $\B^\ddag$ (left) and $\B_0 \cup \B_2$ (right). In the right hand graph, edges in triples in  $\B_0$ are shown in grey}
\end{centering}
\end{figure}

Suppose for a contradiction that there is a nontrivial proper subsystem $(S,\T)$ of $(V,\A \cup \B)$ that is not a subsystem of $(U,\mathcal{A})$. Let $S'=S \cap U$, $S''=S \setminus U$, $s=|S|$, $s'=|S'|$ and $s''=|S''|$. We say a pair or triple of elements of $V$ is \emph{type $i$} if it contains exactly $i$ elements of $U$. Note that each triple in $\A$ is type 3 and, for $i \in \{0,1,2\}$, each triple in $\B_i$ is type $i$. Observe that $s'' > 0$, because $(S,\T)$ is not a subsystem of $(U,\mathcal{A})$ and hence $\T$ must contain a triple not in~$\A$.

We will make use of the following observation. For any point $x \in S'$, each triple in $\T \setminus \B_2$ incident with $x$ contains zero or two pairs in $\{\{x,y\}:y \in S' \setminus \{x\}\}$ and each triple in $\T \cap\B_2$ incident with $x$ contains one of the pairs in $\{\{x,y\}:y \in S' \setminus \{x\}\}$. It follows that each point in $S'$ is in an odd number of triples in $\T \cap \B_2$ if $s'$ is even and each point in $S'$ is in an even number of triples in $\T \cap\B_2$ if $s'$ is odd.

\noindent{\bf Case 1.} Suppose that $\B_0 \subseteq \T$. Then $\{u-1,1,3,5,7\} \subseteq S''$. Hence $\{1,5,\varphi(3)\}, \{5,u-1,\varphi(2)\} \in \T$ by the definitions of $\B$, $\B^\dag$ and $\{F_0,\ldots,F_{u-1}\}$. So $\varphi(2),\varphi(3) \in S$. Note $\varphi(2)=x_3$, $\varphi(3) \in U \setminus \{x_1,\ldots,x_6\}$ and $\{x_3,1,3\}$ is the only triple of $\B^\dag$ containing $x_3$ that is not in $\B$. Thus, for each $x \in \{7,9,\ldots,u\}$, we have that if $x \in S''$, then
\begin{align*}
    \{x,u-x+4,\varphi(2)\} \in \T &\hbox{ and so } u-x+4 \in S'' \hbox{ and}\\
    \{u-x+4,x+2,\varphi(3)\} \in \T &\hbox{ and so } x+2 \in S''
\end{align*}
where all operations take place in $\Z_u$. We know that $7 \in S''$, and so can apply this fact inductively to conclude that $\{7,9,\ldots,u\} \cup \{2\}$ and $\{u-3,u-5,\ldots,4\}$ are subsets of $S''$.
Thus, since $\{u-1,1,3,5,7\} \subseteq S''$, we have $\Z_u \subseteq S''$ and it is easy to conclude from this that $S=V$. This contradicts our assumption that $(S,\T)$ is a proper subsystem of $(V,\A \cup \B)$.

\noindent{\bf Case 2.} Suppose that $|\T \cap \B_0| \in \{0,1\}$.

\noindent{\bf Case 2a.} Suppose that $s'$ is even. Clearly $S' \neq \emptyset$ as otherwise $\T \subseteq \B_0$ and $(S,\T)$ cannot be a nontrivial subsystem. So each point in $S'$ is in an odd number of triples in $\T \cap\B_2$ by our observation. It follows that $S' \subseteq \{x_1,\ldots,x_6\}$ and the subgraph of $H$ induced by $S'$ is a $1$-factor. Thus $s' \in \{2,4\}$. But $s' \neq 4$, for otherwise, counting the type 2 pairs used in triples in $\T$, we would have that $|\T \cap \A|=\frac{1}{3}(6-2)$, contradicting the fact $|\T \cap \A|$ is an integer. Thus $s' = 2$ and $\T$ contains exactly one triple in $\B_2$, say $\{x_i,x_j,k\}$ where $i,j \in \{1,\ldots,6\}$ and $k \in \{u-1,1,3,5,7\}$. Since $s \geq 7$, we have $s'' \geq 5$. But then, counting type 0 pairs involving $k$ used in triples of $\T$, we have that $k$ must be in $\frac{1}{2}(s''-1)\geq 2$ type 0 triples in $\T$. This contradicts the assumption of this case that $\T \cap \B_0 \in \{0,1\}$.

\noindent{\bf Case 2b.} Suppose that $s'$ is odd and $\T \cap \B_2 \neq \emptyset$. Then each point in $S'$ is in an even number of triples in $\T \cap\B_2$ by our observation. This means that the subgraph $H'$ of $H$ induced by $S' \cap \{x_1,\ldots,x_6\}$ contains at least one edge and has no vertices of degree 1. Thus it must be that $H'=H$, $\{x_1,\ldots,x_6\} \subseteq S'$ and $\B_2 \subseteq \T$. But then $\{u-1,1,3,5,7\} \subseteq S''$ and hence $\B_0 \subseteq \T$, contradicting the assumption of this case that $\T \cap \B_0 \in \{0,1\}$.

\noindent{\bf Case 2c.} Suppose that $s'$ is odd and $\T \cap \B_2 = \emptyset$. Then $(S',\T \cap \A)$ is a subsystem of both
$(S,\T)$ and $(U,\A)$ and hence $s' \equiv 1,3 \mod{6}$ and $s''$ is even. Similar pair counting to that used above shows that each $x \in S''$ is in $s'$ type 1 triples in $\T$ and hence in $\frac{1}{2}(s''-s'-1)$ type 0 triples in $\T$. Because $s''$ is even, there is no choice for $S''$ such that each $x \in S''$ is in more than one triple in $\B_0$, and hence it must be the case that $s''=s'+1$ and $\T \cap \B_0 = \emptyset$. So $s' \geq 3$ and $s'' \geq 4$ because $s \geq 7$. Then, by the definitions of $\B$ and $\B^\dag$, $S''$ must induce a sub-1-factorisation of the standard 1-factorisation on $\Z_u \cup \{\infty\}$. So, by Lemma \ref{L:sub1F}, $S''=C \cup \{\infty\}$, where $C$ is a coset of a subgroup of $\Z_u$, and $S'=\varphi(C)$. This contradicts the properties of $\varphi$ because $\varphi(C)=S'$ and $S'$ is the point set of a proper  subsystem of $(U,\A)$.
\end{proof}

\section{Concluding remarks}

Homogeneous structures are a main theme in both model theory and infinite permutation group theory. Fra\"{\i}ss\'e's construction and its variants are rich sources of examples that are central to model-theoretic classification theory. In particular, homogeneous structures omitting certain finite configurations, for instance $K_n$-free graphs~\cite{henson-graph}, their counterparts in higher arities and Henson digraphs~\cite{henson-digraph}, are prominent examples of a range of different model theoretic behaviours. Another recent example of the model theoretic relevance of structures of this kind is~\cite{conant}, which concerns infinite incidence structures omitting the complete incidence structure $K_{m,n}$.

Infinite Steiner systems obtained via Fra\"{\i}ss\'e-type constructions have recently received attention from the model theory community: Barbina and Casanovas~\cite{silvia} give a full model theoretic description of the Fra\"{\i}ss\'e limit of the class of finite Steiner triple systems, whereas Baldwin and Paolini~\cite{baldwin} use Hrushovski's amalgamations -- a variant of Fra\"{\i}ss\'e's construction -- to build uncountably many non-isomorphic Steiner systems that are \emph{strongly minimal}, that is, they are structures whose definable sets are well behaved.

In view of the above, our using Fra\"{\i}ss\'e's construction to produce homogeneous Steiner triple systems that omit certain
subsystems seems timely and natural. The systems we construct in this article can be viewed as analogues of $K_n$-free graphs. The Henson graphs $H_n$ for~$n\geq 3$, that contain all $K_n$-free graphs as subgraphs~\cite{henson-graph},
are one countably infinite family of homogeneous graphs in Lachlan and Woodrow's classification~\cite{LachWood} of all countably infinite homogeneous graphs. A countably infinite graph is homogeneous if and only if it is isomorphic
to one of the following or its complement:
\begin{itemize}
\item
a disjoint union of complete graphs $K_n$ (of the same order);
\item
the Rado graph;
\item
a Henson graph $H_n$ (for some $n \geq 3$).
\end{itemize}

This leads to the following question.

\begin{question}\label{Q:classify}
Is it possible to classify all countably infinite homogeneous locally finite Steiner triple systems?
\end{question}

The fact that there are uncountably many countable homogeneous Steiner triple systems {by Corollary~\ref{C:uncountable} adds to the complexity of such a classification in comparison with countable homogeneous graphs. Cameron~\cite{Cam} noted that this situation brought to his mind Cherlin's classification~\cite{Cherlin} of the uncountably many homogeneous directed graphs, but thought that obtaining a complete solution to Question~\ref{Q:classify} would be extremely difficult. Indeed, we believe there are many other homogeneous Steiner triple systems of countably infinite order beyond those given by Theorem~\ref{T:amalgClass} as we discuss below.

Let $\fF$ be a class of finite partial Steiner triple systems. We say a partial Steiner triple system is $\fF$-free if no isomorphic copy of a partial system in $\fF$ is embedded in it. This extends our earlier definition of $\fF$-free. In this situation, elements of $\fF$ are often referred to as \emph{configurations}. For some classes $\fF$ of configurations (that are not complete systems), the class of all finite $\fF$-free Steiner triple systems is an amalgamation class: the classes of finite projective and affine Steiner triple systems can be characterised as those avoiding certain configurations~\cite{edita}. It seems likely that this is also true for many other classes $\fF$ of configurations. For instance, finite Hall triple systems can also be characterised by the configurations they avoid~\cite{edita,PPPZ} and may form an amalgamation class. As a much less structured example, finite anti-Pasch triple systems (see \cite{GraGriWhi}, for example) may also form an amalgamation class.
Proving results along these lines, however, would require an analogue of Theorem~\ref{T:embedNoNewSubsystems} that guarantees that no new configurations of certain types are created in the embedding process. Steiner triple systems of countably infinite order that contain no ``dense'' configurations are constructed in \cite{CIST}.

\section*{Acknowledgements}

The authors would like to acknowledge the support of the Australian Research Council (grants DP150100506 and FT160100048) and
the Engineering and Physical Sciences Research Council (grant EP/M016242/1). In addition, Bridget S. Webb wishes to acknowledge the support of a 2014 Ethel Raybould Visiting Fellowship from the School of Physical Sciences, The University of Queensland. The authors express their sincere thanks to Silvia Barbina for valuable discussions and advice. Daniel Horsley also thanks Shichang Song for useful discussions.

\end{document}